\documentclass{amsart}
\usepackage{latexsym,amssymb,graphics}

\newtheorem{theorem}{Theorem}

\newtheorem{definition}[theorem]{Definition}
\newtheorem{example}[theorem]{Example}

\newtheorem{remark}[theorem]{Remark}

\input{tcilatex}

\begin{document}
\title[Kronecker and Sidon sets]{The relationship between $\epsilon $%
-Kronecker and Sidon sets}
\author{Kathryn Hare}
\address{Dept. of Pure Mathematics\\
University of Waterloo\\
Waterloo, Ont.,~Canada\\
N2L 3G1}
\email{kehare@uwaterloo.ca}
\thanks{This research is supported in part by NSERC \#44597}
\author{L. Thomas Ramsey}
\address{Dept. of Mathematics\\
University of Hawaii\\
Honolulu, HI 96822}
\email{ramsey@math.hawaii.edu}
\thanks{}
\thanks{This paper is in final form and no version of it will be submitted
for publication elsewhere.}

\begin{abstract}
A subset $E$ of a discrete abelian group is called $\epsilon $-Kronecker if
all $E$-functions of modulus one can be approximated to within $\epsilon $
by characters. $E$ is called a Sidon set if all bounded $E$-functions can be
interpolated by the Fourier transform of measures on the dual group. As $%
\epsilon $-Kronecker sets with $\epsilon <2$ possess the same arithmetic
properties as Sidon sets, it is natural to ask if they are Sidon. We use the
Pisier net characterization of Sidonicity to prove this is true.
\end{abstract}

\maketitle

\section{Introduction}

A subset $E$ of the dual of a compact, abelian group $G$ is called an $%
\epsilon $\textit{-Kronecker set} if for every function $\phi $ mapping $E$
into the set of complex numbers of modulus one there exists $x\in G$ such
that 
\begin{equation*}
\left\vert \phi (\gamma )-\gamma (x)\right\vert <\epsilon \text{ for all }%
\gamma \in E\text{.}
\end{equation*}%
The infimum of such $\epsilon $ is called the \textit{Kronecker constant }of 
$E$ and is denoted $\kappa (E)$. Trivially, $\kappa (E)\leq 2$ for all sets $%
E$ and this is sharp if the identity of the dual group belongs to $E$. $%
\epsilon $-Kronecker sets have been studied for over 50 years since the
concept was introduced by Kahane in \cite{K} and the terminology was coined
by Varapoulos in \cite{Va}.  Examples of recent work include \cite{GH}, \cite%
{GK} (where they are called $\epsilon $-free), \cite{GH1}-\cite{HRThree} and 
\cite{KR}.

If $\kappa (E)<\sqrt{2}$, then $E$ is known to be an example of a Sidon set,
meaning every bounded $E$-function is the restriction to $E$ of the Fourier
transform of a measure on $G$. In fact, the interpolating measure can be
chosen to be discrete and $\sqrt{2}$ is sharp with this additional property.
Like $\epsilon $-Kronecker sets, Sidon sets have also been extensively
studied for many years; we refer the reader to \cite{HewR} or \cite{LR} for
an overview of what was known prior to the early 1970's and to \cite{GH13}
for more recent results. But many fundamental problems remain open,
including a full understanding of the connections between these two classes
of interpolation sets.

As sets with Kronecker constant less than $2$ possess many of the known
arithmetic properties satisfied by Sidon sets, it was asked in \cite{GH13}
whether all such sets are Sidon. Here we answer this question affirmatively
by using Pisier's remarkable net characterization of Sidon sets. We also
construct non-trivial examples of Sidon sets with Kronecker constant $2$.

As well, we define a weaker interpolation property than $\epsilon $%
-Kronecker by only requiring the approximation of target functions whose
range lies in the set of $n$'th roots of unity. Sets that satisfy a suitable
quantitative condition for this less demanding interpolation property are
also shown to be Sidon.

\section{Kronecker-like sets that are Sidon}

Let $G$ be a compact abelian group and $\Gamma $ its discrete abelian dual
group. An example of such a group $G$ is the circle group $\mathbb{T}$, the
complex numbers of modulus one, whose discrete dual is the group of
integers, $\mathbb{Z}$.

\begin{definition}
(i) A subset $E\subseteq \Gamma $ is said to be $\epsilon $\textbf{-Kronecker%
} if for every $\phi :E\rightarrow \mathbb{T}$ there exists $x\in G$ such
that 
\begin{equation}
\left\vert \phi (\gamma )-\gamma (x)\right\vert <\epsilon \text{ for all }%
\gamma \in E\text{.}  \label{Kronecker}
\end{equation}%
By the \textbf{Kronecker constant }of $E$, $\kappa (E)$, we mean the infimum
of the constants $\epsilon $ for which (\ref{Kronecker}) is satisfied.

(ii) A subset $E\subseteq \Gamma $ is said to be \textbf{Sidon} if for every
bounded function $\phi :E\rightarrow \mathbb{C}$ there is a measure $\mu $
on $G$ with $\widehat{\mu }(\gamma )=\phi (\gamma )$ for all $\gamma \in E$.
If the interpolating measure $\mu $ can always be chosen to be discrete,
then the set $E$ is said to be\textbf{\ }$I_{0}$.
\end{definition}

Hadamard sets $E=\{n_{j}\}\subseteq \mathbb{N}$ with $\inf n_{j+1}/n_{j}=q>2$
are known to satisfy $\kappa (E)\leq $ $\left\vert 1-e^{i\pi
(q-1)}\right\vert $ and this tends to $0$ as $q$ tends to infinity. More
generally, every infinite subset of a torsion-free dual group $\Gamma $
contains subsets of the same cardinality that are $\epsilon $-Kronecker for
any given $\epsilon >0$. If $\Gamma $ is not torsion free, but the subset $E$
does not contain `too' many elements of order 2, then $E$ will contain a
subset $F$ of the same cardinality, having $\kappa (F)=$ $1$ (see \cite{GH1}%
, \cite{GHColloq}).

Obviously, every $I_{0}$ set is Sidon, but the converse is not true. It is
unknown whether every Sidon set is a finite union of $I_{0}$ sets.

For a set $E$ to be Sidon (or $I_{0}$) it is enough that there be a constant 
$\delta <1$ such that for every $E$-function $\phi $ with $\left\vert \phi
(\gamma )\right\vert \leq 1$ for all $\gamma $, there is a (discrete)
measure $\mu $ such that 
\begin{equation*}
\left\vert \phi (\gamma )-\widehat{\mu }(\gamma )\right\vert <\delta \text{
for all }\gamma \in E\text{.}
\end{equation*}%
Since $\gamma (x)=\widehat{\delta _{x}}(\gamma )$ for $\delta _{x}$ the
point mass measure at $x$, it is easy to see that if $E$ is $\epsilon $%
-Kronecker for some $\epsilon <1$, then $E$ is $I_{0}$. With more work this
can be improved: If $\kappa (E)<\sqrt{2},$ then $E$ is $I_{0}$. This result
is sharp as there are $I_{0}$ sets that are $\sqrt{2}$-Kronecker; see \cite%
{GH1}.

It is well known that Sidon sets satisfy a number of arithmetic properties,
such as not containing large squares or long arithmetic progressions. In 
\cite{GH1} (or see the discussion in \cite[p. 35]{GH13}), it was shown that
sets $E$ with $\kappa (E)<2$ also satisfy these conditions, thus it is
natural to ask if such sets are always Sidon. Here we answer this question
affirmatively.

\begin{theorem}
\label{T:prob4} If the Kronecker constant of $E\subseteq \Gamma $ is less
than two, then $E$ is Sidon.
\end{theorem}

\begin{proof}
We shall use Pisier's $\varepsilon $-net condition which states that a
subset $E$ is Sidon if and only if there is some $\varepsilon >0$ such that
for each finite subset $F\subset E$ there is a set $Y\subset G$ with $%
|Y|\geq 2^{\varepsilon |F|}$ and whenever $x\neq y\in Y$, then 
\begin{equation*}
\varepsilon \leq \sup_{\gamma \in F}\left\vert \gamma (x)-\gamma
(y)\right\vert .
\end{equation*}%
This was proven by Pisier in \cite{Pi}. Proofs can also be found in \cite[%
Thm. 9.2.1]{GH13} and \cite[Thm. V.5]{LQ}.

Since we are assuming $\kappa (E)<2$, we can choose $\varepsilon >0$ such
that $\kappa (E)+\varepsilon <2$. Let $F$ be any finite subset of $E$.

For all $g\in G$ and $\lambda >0$, the sets 
\begin{equation*}
U(g,\lambda )=\left\{ \,{h\in G}\,:\,{\lambda >\sup_{\gamma \in F}|\gamma
(h)-\gamma (g)|}\,\right\}
\end{equation*}%
are among the basic open sets for the topology on $G$ (the topology of
pointwise convergence as functions on $\Gamma $). We claim there is a finite
maximal set $S\subset G$ such that 
\begin{equation*}
x\neq y\in S\quad \Rightarrow \quad \varepsilon \leq \sup_{\gamma \in
F}|\gamma (x)-\gamma (y)|.
\end{equation*}%
This is a consequence of the compactness of $G$. If it was not true, one
could choose an infinite set $S$ having this separation property. As $G$ is
compact, $S$ would have a cluster point $z\in G$. The open set $%
U(z,\varepsilon /2)$ would then contain infinitely many members of $S$,
violating the required separation assumption.

By the maximality of $S$, for each $g\in G$ there is some $h\in S$ such that 
$g\in U(h,\varepsilon )$.

Consider any function $\phi :F\rightarrow \mathbb{T}$. By the Kronecker
property, there is some $g\in G$ such that $\sup_{\gamma \in F}\left\vert
\gamma (g)-\phi (\gamma )\right\vert \leq \kappa (E)$. Since there is some $%
h\in S$ such that $g\in U(h,\epsilon ),$ we have that $\phi \in W(h)$ where 
\begin{equation*}
W(h):=\,\left\{ \psi {:F\rightarrow }\mathbb{T}\,:\,{\sup_{\gamma \in F}}%
\left\vert {\gamma (h)-\psi (\gamma )}\right\vert \leq {\kappa
(E)+\varepsilon <2}\right\} .
\end{equation*}%
Consequently,%
\begin{equation*}
\mathbb{T}^{F}=\bigcup_{h\in S}W(h).
\end{equation*}

We identify $\mathbb{T}^{F}$ with $[0,2\pi )^{F},$ with the group operation
being addition $\func{mod}2\pi ,$ and in this way put $\left\vert
F\right\vert $-dimensional Euclidean volume on $\mathbb{T}^{F}$. With this
identification, 
\begin{equation*}
W(h)\subseteq \tprod\limits_{\gamma \in F}[\gamma (h)-\eta ,\gamma (h)+\eta
],
\end{equation*}%
where $\eta <\pi $ depends only on the number $\kappa (E)+\varepsilon $ (and
not on $h$ or $F$). Thus the $\left\vert F\right\vert $-dimensional volume
of each set $W(h)$ is bounded by $(2\eta )^{|F|}$, while the volume of $%
\mathbb{T}^{F}$ is $(2\pi )^{\left\vert F\right\vert }$. It follows that 
\begin{equation*}
card(S)\geq \left( \dfrac{2\pi }{2\eta }\right) ^{|F|}=2^{\varepsilon
^{\prime }|F|}
\end{equation*}%
for a suitable choice of $\varepsilon ^{\prime }>0$.

The minimum of $\varepsilon $ and $\varepsilon ^{\prime }$ meet the Pisier
net condition and are independent of $F$. Thus $E$ is Sidon.
\end{proof}

\begin{remark}
In number theory, a set $E\subseteq \Gamma $ is sometimes called a Sidon set
if whenever $\gamma _{j}\in E$, then $\gamma _{1}\gamma _{2}=\gamma
_{3}\gamma _{4}$ if and only if $\{\gamma _{3},\gamma _{4}\}$ is a
permutation of $\{\gamma _{1},\gamma _{2}\}$. This is a different class of
sets from the Sidon sets defined above. $\varepsilon $-Kronecker sets need
not be Sidon in this sense; indeed any finite subset $E\subseteq $ $\mathbb{Z%
}$ that does not contain $0$ has $\kappa (E)<2$. However, if $E$ is $%
\varepsilon $-Kronecker for some $\varepsilon <\sqrt{2}$, then there are a
bounded number of pairs with common product, with the bound depending only
on $\varepsilon $ (see \cite{GH1}).
\end{remark}

Next, we alter the definition of the Kronecker constant by only considering
target functions whose range is restricted to a finite subgroup of $\mathbb{T%
}$. This is a natural variation to consider for if $\Gamma $ is a torsion
group the characters of $G$ take on only the values in a suitable finite
subgroup of $\mathbb{T}$. Moreover, there are even subsets $E$ of $\mathbb{Z}
$ (including all subsets of size 2 and many of size 3) whose Kronecker
constant is realized with target functions $\phi $ mapping $E$ into $%
\{-1,+1\}$ (c.f. \cite{HRThree}).

\begin{definition}
Let $\mathbf{T}_{n}$ denote the set of $n$'th roots of unity in $\mathbb{T}$
for $n\geq 2$. Let $\kappa _{n}(E)$ be the infimum of $\epsilon \geq 0$ such
that $E$ is $(\epsilon ,n)$-Kronecker, where $E\subseteq \Gamma $ is $%
(\epsilon ,n)$-Kronecker if for every $\phi :E\rightarrow \mathbf{T}_{n}$
there exists $x\in G$ such that 
\begin{equation*}
\gamma \in E\quad \Rightarrow \quad \left\vert \phi (\gamma )-\gamma
(x)\right\vert <\epsilon .
\end{equation*}
\end{definition}

\begin{theorem}
\label{T:prob4n} Let $E\subset \Gamma $. If $\kappa _{n}(E)<\left\vert
1-e^{i\pi (1-1/n)}\right\vert $, then $E$ is Sidon.
\end{theorem}

\begin{proof}
Choose $\varepsilon >0$ such that $\kappa _{n}(E)+\varepsilon <\left\vert
1-e^{i\pi (1-1/n)}\right\vert $. Let $F\subset E$ be finite. Choose $%
S\subset G$ as in the proof of Theorem \ref{T:prob4}. Arguing in a similar
fashion to that proof, we again deduce that for every $\phi :E\rightarrow 
\mathbf{T}_{n}$ there is some $h\in S$ such that $\phi \in V(h),$ where 
\begin{equation*}
V(h):=\left\{ \psi {:F\rightarrow \mathbf{T}_{n}}\,:\,{\sup_{\gamma \in F}}%
\left\vert {\gamma (h)-\psi (\gamma )}\right\vert \leq {\kappa
_{n}(E)+\varepsilon }\,\right\} .
\end{equation*}%
Consequently 
\begin{equation*}
(\mathbf{T}_{n})^{F}=\bigcup_{h\in S}V(h).
\end{equation*}%
%
%
%
%
%
%
%
%
%
For each $h\in S$ and every $\gamma \in F,$ there is an $n$'th root of
unity, $\omega _{n}\in $ $\mathbf{T}_{n},$ such that $\left\vert \gamma
(h)-\omega _{n}\right\vert \geq \left\vert 1-e^{i\pi (1-1/n)}\right\vert $.
Defining $\phi _{h}(\gamma )=\omega _{n}$, it follows that $\phi _{h}\notin
V(h)$. Thus each $V(h)$ has at most $(n-1)^{|F|}$ elements. Consequently,
there is some $\varepsilon ^{\prime }>0$, independent of $F$, such that 
\begin{equation*}
card(S)\geq \dfrac{n^{|F|}}{(n-1)^{|F|}}=2^{\epsilon ^{\prime }|F|}.
\end{equation*}%
Again, the minimum of $\varepsilon $ and $\varepsilon ^{\prime }$ meets the
Pisier net condition to be Sidon.
\end{proof}

It is sometimes more convenient to measure angular distances when comparing
elements of $\mathbb{T}$ and to express Kronecker constants in those terms.
Towards this, put $\mathbf{Z}_{n}=\,\left\{ 2\pi j/n:{j=0,1,...,n-1}\right\} 
$ and for $z\in \mathbb{T}$, let $\arg (z)$ be the angle $\theta \in \lbrack
0,2\pi )$ such that $\exp (i\theta )=z$. Let $\alpha _{n}(E)$ be the infimum
of $\epsilon \geq 0$ such that for every $\phi :E\rightarrow \mathbf{Z}_{n}$
there exists $x\in G$ such that 
\begin{equation*}
\gamma \in E\quad \Rightarrow \quad \left\vert \phi (\gamma )-\arg \gamma
(x)\right\vert \leq \epsilon .
\end{equation*}%
A set $E$ satisfying this condition is called weak $(\epsilon ,n)$-angular
Kronecker. Here $|\phi (\gamma )-\arg \gamma (x)|$ should be understood $%
\func{mod}2\pi ,$ so $\alpha _{n}(E)\in \lbrack 0,\pi ].$

It is easy to see that $\kappa _{n}(E)=\left\vert 1-e^{i\alpha
_{n}(E)}\right\vert $, thus the previous theorem can be restated as: $E$ is
Sidon if $\alpha _{n}(E)<\pi (1-1/n)$.

We can similarly define weak angular $\epsilon $-Kronecker sets and the
angular Kronecker constant, $\alpha (E),$ by considering the approximation
problem for functions $\phi :E\rightarrow \lbrack 0,2\pi )$. One can easily
check that $\kappa (E)=\left\vert 1-e^{i\alpha (E)}\right\vert $, hence
Theorem \ref{T:prob4} can be restated as: $E$ is Sidon if $\alpha (E)<\pi $.

\begin{example}
Let $n>1$ be any integer. The set $E=1+n\mathbb{Z}$ is not a Sidon subset of 
$\mathbb{Z}$ being a coset of an infinite subgroup, but $\alpha _{n}(E)=\pi
-\pi /n$. That shows Theorem \ref{T:prob4n} is sharp. In fact, for odd $n$, $%
\alpha _{n}(E)\leq \pi -\pi /n$ for all subsets $E$ of any discrete abelian
group $\Gamma $. This is because the $n$'th root of unity farthest from $1$
is $e^{i\pi (1-1/n)}$, so that if we let $1$ denote the identity element of $%
G$, then for all $\mathbf{T}_{n}$-valued functions $\phi $ and any $\gamma
\in \Gamma $ we have $|\phi (\gamma )-\gamma (1)|\leq \left\vert 1-e^{i\pi
(1-1/n)}\right\vert $.

To see that $\alpha _{n}(1+n\mathbb{Z})\leq $ $\pi -\pi /n$ for $n$ even,
take $g=\exp (\pi i/n)$. For any character $\gamma =1+nk\in E$, we have $%
\arg \gamma (g)=\pi (nk+1)/n$ with $nk+1$ an odd integer. Thus $|z-\arg
\gamma (x)|\leq \pi -\pi /n$ for any $z\in \mathbf{Z}_{n}$.
\end{example}

\section{Some Examples of Sidon Sets with Kronecker Constant Equal to $2$}

Since any subset of $\Gamma $ that contains the identity character $1$ has
Kronecker constant equal to $2$, we are interested in constructing examples
of Sidon subsets $E$ of $\Gamma \diagdown \{1\}$ with $\kappa (E)=2$ and $%
\kappa _{n}(E)\geq \left\vert 1-e^{i\pi (1-1/n)}\right\vert $. We give one
example with a set of elements of finite order and a second example where
all the elements of $E$ have infinite order.

\begin{example}
Let $\Gamma =\mathbb{Z}_{2}\oplus \mathbb{Z}_{2}$ $\oplus \mathbb{Z}_{2}$,
where $\mathbb{Z}_{2}=\{0,1\}$. Then $\Gamma \diagdown (0,0,0)$ is Sidon,
but $\kappa (E)=2$ and $\kappa _{n}(E)\geq \left\vert 1-e^{i\pi
(1-1/n)}\right\vert $ for $n\geq 2$.
\end{example}

\begin{proof}
Being a finite set, $\Gamma \diagdown (0,0,0)$ is Sidon. Let $e_{j}$ be the
standard basis vectors of $\mathbb{Z}_{2}\oplus \mathbb{Z}_{2}$ $\oplus 
\mathbb{Z}_{2}$ and let $E=\{e_{2},e_{3},e_{1}+e_{2},e_{1}+e_{3}\}$.

We will first show that $\kappa (E)=2$. Define $\phi $ by $\phi (e_{2})=\phi
(e_{3})=\phi (e_{1}+e_{2})=1$ and $\phi (e_{1}+e_{3})=-1$. Suppose that $%
g\in G$ and $\epsilon >0$ satisfies 
\begin{equation*}
|\gamma (g)-\phi (g)|<2-\epsilon \quad \text{for all $\gamma \in E.$}
\end{equation*}%
Because $\gamma (g)\in \{-1,+1\}$ for every $\gamma \in \Gamma $, we must
have $e_{2}(g)=e_{3}(g)=1=(e_{1}+e_{2})(g)$ and $(e_{1}+e_{3})(g)=-1$. This
forces $e_{1}(g)$ to be equal to both $-1$ and $1$, a contradiction. Hence $%
\kappa (E)=2$.

Since $\phi $ takes on only $n$'th roots of unity for even $n$, this
argument also proves $\kappa _{n}(E)=2$ when $n$ is even.

If $n$ is odd then, instead, define $\phi (e_{1}+e_{3})=\omega _{n}$, where $%
\omega _{n}=e^{i\pi (1-1/n)}$, an $n$'th roots of unity nearest to $-1$. If $%
\kappa _{n}(E)<\left\vert 1-e^{i\pi (1-1/n)}\right\vert $, then we obtain
the same contradition as before by noting that the identity $\left\vert
1-\phi (e_{1}+e_{3})\right\vert =\left\vert 1-e^{i\pi (1-1/n)}\right\vert $
forces $(e_{1}+e_{3})(g)=-1.$
\end{proof}

\begin{example}
Let $\Gamma =\mathbb{Z}$ $\oplus \Gamma _{2}$ where $\Gamma _{2}$ is the
countable direct sum of copies of $\mathbb{Z}_{2}$. Let $e_{n}$ be the
character $e^{2\pi in(\cdot )}$ on $\mathbb{T}$ and let $\gamma _{n}$ be the
projection onto the $n$'th-$\mathbb{Z}_{2}$ factor, both viewed as elements
of $\Gamma $ in the canonical way. Set 
\begin{equation*}
E=\{(e_{n},\gamma _{n})\}_{n=1}^{\infty }\bigcup \{(e_{n}^{-1},\gamma
_{n})\}_{n=1}^{\infty }.
\end{equation*}%
Then $E$ is Sidon, but $\kappa (E)=2$ and $\kappa _{n}(E)\geq \left\vert
1-e^{i\pi (1-1/n)}\right\vert $ for $n\geq 2$.
\end{example}

\begin{proof}
We argue first that $E_{1}=\{(e_{n},\gamma _{n})\}_{n=1}^{\infty }$ and $%
E_{2}=\{(e_{n}^{-1},\gamma _{n})\}_{n=1}^{\infty }$ both satisfy algebraic
conditions to be Sidon. Let $f:E_{1}\rightarrow \{-1,0,1\}$ be finitely
non-zero and satisfy By the algebraic independence of the factors of $\Gamma 
$ this implies $\gamma _{n}^{f(n)}=1$ for all $n$ and hence $f(n)=0$.
Therefore $E_{1}$ is quasi-independent and such sets are well known to be
Sidon. Likewise, $E_{2}$ is Sidon and hence the union, $E=E_{1}\cup E_{2},$
is Sidon.

Let $\epsilon >0$ and suppose $E$ is $(2-\epsilon )$-Kronecker. Define $\phi 
$ to be $-1$ on $E_{1}$ and $1$ on $E_{2}$. The compact group $G=\mathbb{%
T\otimes }G_{2}$ where $G_{2}$ is the direct product of countably many
copies of (the multiplicative group) $\mathbb{Z}_{2}$ is the dual of $\Gamma 
$. Choose $g\in G$ such that for all $\gamma \in E$, 
\begin{equation*}
\left\vert \phi (\gamma )-\gamma (g)\right\vert <2-\epsilon .
\end{equation*}

Write $g=(u,(g_{n}))$ where $u\in \mathbb{T}$ and $g_{n}$ is the projection
of $g$ onto the $n$'th-$\mathbb{Z}_{2}$ factor. With this notation, $%
(e_{n}^{\pm 1},\gamma _{n})(g)=e^{\pm 2\pi inu}g_{n}$, hence for all $n$, 
\begin{eqnarray*}
\left\vert -e^{-2\pi inu}-g_{n}\right\vert &=&|-1-e^{2\pi
inu}g_{n}|<2-\epsilon \quad \text{and}\quad \\
\left\vert e^{2\pi inu}-g_{n}\right\vert &=&|1-e^{-2\pi
inu}g_{n}|<2-\epsilon \text{.}
\end{eqnarray*}

If $u$ is rational, then $e^{2\pi inu}=e^{-2\pi inu}=1$ periodically as a
function of $n$. For these infinitely many $n$, we have $|-1-g_{n}|<2-%
\epsilon $ and $|1-g_{n}|<2-\epsilon $. But $g_{n}=\pm 1$, so this is
impossible.

Otherwise, $\{e^{2\pi inu}\}_{n=1}^{\infty }$ is dense in $\mathbb{T}$.
Choose $n$ such that 
\begin{equation*}
\left\vert 1-e^{2\pi inu}\right\vert =\left\vert 1-e^{-2\pi inu}\right\vert
<\epsilon /2.
\end{equation*}
But then $|-1-g_{n}|<2-\epsilon /2$ and $|1-g_{n}|<2-\epsilon /2$, and again
these cannot be simultaneously satisfied for $g_{n}=\pm 1$. This
impossibility proves $\kappa (E)=2$ and also establishes $\kappa _{n}(E)=2$
for $n$ even.

If, instead, we define $\phi =\omega _{n}$ on $E_{1}$, where $\omega _{n}$
is an $n$'th root of unity nearest $-1$, then similar arguments show that $%
\kappa _{n}(E)=\left\vert 1-e^{i\pi (1-1/n)}\right\vert $ for $n$ odd and $%
\kappa _{n}(E)=2$ for $n$ even.
\end{proof}

\begin{remark}
It would be interesting to know whether non-trivial examples of $2$%
-Kronecker, Sidon sets could be found in a torsion-free group and also
whether every Sidon set is a finite union of sets that are $\epsilon $%
-Kronecker for some $\epsilon <2$.
\end{remark}


\begin{thebibliography}{99}
\bibitem{GH} J. Galindo and S. Hernandez, \textit{The concept of boundedness
and the Bohr compactification of a MAP abelian group}, Fund. Math. \textbf{15%
}(1999), 195-218.

\bibitem{GK} B.N. Givens and K. Kunen, \textit{Chromatic numbers and Bohr
topologies}, Top. Appl. \textbf{131}(2003), 189-202.

\bibitem{GH1} C.C. Graham and K.E. Hare, $\varepsilon $\textit{-Kronecker
and }$I_{0}$\textit{\ sets in abelian groups, I: arithmetic properties of }$%
\varepsilon $\textit{-Kronecker sets,} Math. Proc. Camb. Philo. Soc. \textbf{%
140}(2006), 475-489.

\bibitem{GHColloq} C.C. Graham and K.E. Hare, \textit{Existence of large }$%
\varepsilon $\textit{-Kronecker and }$FZI_{0}$\textit{\ sets in discrete
abelian groups}, Colloq. Math. \textbf{127}(2012), 1-15.

\bibitem{GH13} C.C. Graham and K.E. Hare, \textit{Interpolation and Sidon
Sets for Compact Groups}, Springer-Verlag, New York, 2013.

\bibitem{GL} C.C. Graham and A. T-M. Lau, \textit{Relative weak compactness
of orbits in Banach spaces associated with locally compact groups}, Trans.
Amer. Math. Soc. \textbf{359}(2007), 1129-1160.

\bibitem{HRThree} K.E. Hare and L.T. Ramsey, \textit{Kronecker constants of
three element sets, }to appear Acta. Math. Hung., 2015.

\bibitem{HewR} E. Hewitt and K.A. Ross, \textit{Abstract harmonic analysis,
volume II}, Springer-Verlag, New York, 1970.

\bibitem{K} J-P. Kahane, \textit{Algebres tensorielles et analyse harmonique}%
, Seminaire Bourbaki, 1964-6, Exposes 277-312, Soc. Math. France, Paris,
1995, 221-230.

\bibitem{KR} K. Kunen and W. Rudin,\textit{\ Lacunarity and the Bohr topology%
}, Math. Proc. Camb. Phil. Soc. \textbf{126}(1999), 117-137.

\bibitem{LQ} D. Li and H. Queffelec, \textit{Introduction a l'etude des
espaces de Banach analyse et probabilites}, Soc. Math. de France, Paris,
2004.

\bibitem{LR} J. Lopez and K.A. Ross,\textit{\ Sidon sets}, Lecture notes in
Pure and Appl. Math. \textbf{13}, Marcel Dekker, New York, N.Y. 1975.

\bibitem{Pi} G. Pisier, \textit{Conditions d'entropie et caracterisations
arithmetique des ensembles de Sidon}. In Proc. Conf. on modern topics in
harmonic analysis, Inst. de Alta Math, Torino, 1982, 911-941.

\bibitem{Va} N. Varapoulos, \textit{Tensor algebras and harmonic analysis},
Acta Math. \textbf{119}(1968), 51-112.
\end{thebibliography}
\end{document}